\documentclass[11pt,twoside]{amsart}

\usepackage{amsmath}
\usepackage{amsthm}
\usepackage{amsfonts, amssymb}
\usepackage{mathrsfs}
\usepackage[all,line]{xy}
\usepackage{url}
\usepackage{graphics}
\usepackage{epstopdf}

\usepackage{latexsym}
\usepackage[dvips]{graphicx}

\theoremstyle{plain}
\newtheorem{lemma}{Lemma}[section]

\newtheorem{thm}[lemma]{Theorem}
\newtheorem*{thm1}{Theorem \ref{main}}
\newtheorem*{thm2}{Theorem \ref{relative}}

\theoremstyle{remark}

\newtheorem{remark}[lemma]{Remark}

\theoremstyle{definition}
\newtheorem{definition}[lemma]{Definition}
\newtheorem{ej}[lemma]{Example}

\def\u#1{U_{#1}}

\def\ut#1{\hat{U}_{#1}}

\def\set#1{\{#1\}}
\def\k{\mathcal{K}}
\def\x{\mathcal{X}}
\def\h{\mathcal{H}}

\def\R{\mathbb{R}}

\def\ti#1{\tilde{#1}}
\begin{document}

\title[Morse theory of Bestvina-Brady type]{Morse theory of Bestvina-Brady type for posets and matchings}
\author[E.G. Minian]{El\'\i as Gabriel Minian}

\address{Departamento  de Matem\'atica-IMAS\\
 FCEyN, Universidad de Buenos Aires\\ Buenos
Aires, Argentina.}

\email{gminian@dm.uba.ar}

\thanks{Researcher of CONICET. Partially supported by grants PICT 2017-2997, PICT 2019-2338 and UBACYT 20020190100099BA}

\begin{abstract}
We introduce a Morse theory for posets of Bestvina-Brady type combining matchings and height functions. This theory generalizes Forman's discrete Morse theory for regular CW-complexes and extends previous results on Morse theory for $h$-regular posets to all finite posets. We also develop a relative version of Morse theory which allows us to compare the topology of a poset with that of a given subposet. 
\end{abstract}



\maketitle

\section{Introduction}

Classical Morse theory is a tool that allows us to study the topology of differentiable manifolds by analyzing smooth real-valued functions. For instance, the homotopy type of a manifold $M$ can be described in terms of the critical points of a Morse function $f:M\to\R$. More concretely, $M$ has the homotopy type of a CW-complex with one cell of dimension $k$ for each critical point of index $k$. Applications of this theory in mathematics and physics abound. Just to name a few: it is used in in the classification of manifolds, in Bott periodicity \cite{bott} and in the h-cobordism theorem  \cite{mil2}. An excellent reference for classical Morse theory is Milnor's book \cite{mil1}. In the nineties Forman introduced a discrete version of Morse theory for regular CW-complexes, described in terms of discrete analogues of Morse functions and gradient vector fields \cite{for,for2,for3}. In Forman's setting, the Morse functions assign a real number to each cell of the regular complex. There are many applications of Forman's theory in combinatorics, topology, algebra, dynamical systems and topological data analysis (see for example \cite{bw,mn,mro,sha}). It is well known that Forman's theory for regular CW-complexes can  be formulated in terms of acyclic matchings on the Hasse diagrams of the associated face posets. This is an observation of Chari \cite{cha}, which is essentially a combinatorial description of the gradient vector field associated to a discrete Morse function (see also \cite{for3}). A Morse matching on the Hasse diagram of the face poset of a regular complex encodes the same information as the Morse function on the complex. In \cite{min} we extended Forman's theory to $h$-regular posets, adapting Chari's construction to a more general context, although in the non-regular setting one cannot use that the edges of the matchings correspond to collapses.

Bestvina-Brady version of Morse theory applies to affine cell complexes \cite{bb}. In this case, a Morse function is a PL-function defined on the cell complex and the critical points are the vertices of the complex. To understand how the topology of the complex changes when a critical level is crossed, one studies the descending links of the vertices (these are the directions in which the function decreases). In contrast to the classical case of manifolds, where the decreasing directions form a sphere, in Bestvina-Brady theory the descending links may have different homotopy types. When crossing a critical level, the homotopy type changes by adding cones on the descending links of the critical points at that level. The first powerful applications of this theory are related to topological and cohomological aspects of groups \cite{bb}, but there are also applications in combinatorics, algebra, topology, geometric group theory and topological data analysis (see for example \cite{bes,bux,zar}). A very nice survey is Bestvina's article \cite{bes}. 
Recently Zaremsky observed that Forman's theory can be seen as a particular case of Bestvina-Brady Morse theory  \cite{zar}. Concretely, if $f$ is a Forman's Morse function defined on the cells of regular CW-complex $L$, which we can take injective, the PL function $L'\to\R$ induced on the barycentric subdivision $L'$, that coincides with $f$ on the vertices of $L'$ (=cells of $L$), is a Morse function in the sense of Bestvina-Brady. Moreover, in that case, the descending links of the critical vertices of $L'$ (i.e. the critical cells of $L$) are spheres, and the non-critical cells of $L$, viewed as vertices of $L'$, have contractible descending links. This implies Forman's result, since gluing a cone on a contractible space does not change the homotopy type and gluing a cone on a sphere corresponds to attaching a cell (see \cite{zar} for more details). Some advantages of Forman's theory are its simplicity and that it allows to find, via matchings (or discrete gradient fields), Morse functions with few critical points. On the other hand, Bestvina-Brady's theory can be a applied in a more general setting.

In this paper we combine both approaches to extend Forman's theory to all finite posets (not only regular or $h$-regular).  We use height-type functions and admissible Morse matchings to describe and analyze the homotopy type of the order complex of any finite poset in terms of gluing cones on the descending links of the non-critical points. The matching is used to discard the points whose descending links are contractible and the function provides a convenient ordering for gluing the cones. When the poset is regular, i.e. if it is the face poset of a regular CW-complex, and the Morse function is the usual height function of the poset, we recover Forman's theory. If the poset is $h$-regular (and the function is the height function) we obtain our previous result of \cite{min}. In the general case, the homotopy types of the descending links may be contractible, models of spheres, or have different homotopy types.
 
Our main result is the following.
\begin{thm1}
	Let $f:X\to \R$ be a Morse function and let $M$ be an admissible Morse matching for $f$. Then the order complex $\k(X)$ is homotopy equivalent to a CW-complex constructed by inductively gluing cones over the (order complexes of) descending links of the critical points. Concretely, if $t_0<t_1<\ldots<t_r \in\R$ are the values of $f$ and $C_i$ denotes the set of critical points $x\in X$ with $f(x)=t_i$, there is a filtration of $X$ by subposets $X_0\subseteq \ti X_0\subseteq X_1\subseteq \ldots \subseteq X_i\subseteq \ti X_i\subseteq \ldots \subseteq X_r=X$, where $\k(X_i)=\k(\ti X_{i-1})\bigcup_{x \in C_i} x\k(\ut x)$, and $\k(X_i)\subseteq \k(\ti X_i)$ is a strong deformation retract for all $i$.
\end{thm1}
In particular, $\k(X_i)$ is homotopy equivalent to $\k(X_{i-1})\bigcup_{x \in C_i} x\k(\ut x)$, where the cone  $x\k(\ut x)$ over the descending link of each critical point is glued via the deformation retraction $\k(X_{i-1})\subseteq \k(\ti X_{i-1})$. We include various examples to illustrate how our theory applies. We also develop a relative version of the theory which compares the topology of a poset with that of a given subposet.

 \begin{thm2}
 	Let $A\subset X$ be a down-set, $f:A^c\to \R$ a Morse function, and $M$ an admissible Morse matching for $f$. Then the order complex $\k(X)$ is homotopy equivalent to a CW-complex constructed from $\k(A)$ by inductively gluing cones over the descending links of the critical points. Concretely, if $t_0<t_1<\ldots<t_r \in\R$ are the values of $f$ and $C_i$ denotes the set of critical points $x\in A^c$ with $f(x)=t_i$, there is a filtration of $X$ by subposets
 	$A=\ti X_{-1}\subseteq X_0\subseteq \ti X_0\subseteq X_1\subseteq \ldots \subseteq X_i\subseteq \ti X_i\subseteq \ldots \subseteq X_r=X$, where $\k(X_i)=\k(\ti X_{i-1})\bigcup_{x \in C_i} x\k(\ut x)$, and $\k(X_i)\subseteq \k(\ti X_i)$ is a strong deformation retract for all $i$.
 \end{thm2}  

In this paper, all the posets and CW-complexes that we deal with, are assumed to be finite.

\section{Preliminaries}

Let $X$ be a poset. Given $x,y\in X$, we write $x\prec y$ if $x$ is covered by $y$, i.e. if $x<y$ and there is no $z\in X$ such that $x<z<y$. Recall that the Hasse diagram $\h(X)$ of a poset $X$ is the digraph whose vertices are the points of $X$ and whose edges are the pairs $(x,y)$ such that $x\prec y$. When we represent $\h(X)$ graphically, instead of writing  the edge $(x,y)$ with an arrow from $x$ to $y$, we simply put $y$ over $x$ (see Figure \ref{nuevo}). The height $h(X)$ of a poset $X$ is the maximum of the lengths of the chains in $X$, where the length of a chain $x_0<x_1<\ldots<x_n$ is $n$. The height $h(x)$ of an element $x\in X$ is the height of the subposet $\u x=\{y\in X, \  y\le x\}$.

Any poset has an intrinsic topology, where the open sets are the down-sets (see for example \cite{b1,bm1,mcc,sto}). Recall that a subposet $U\subseteq X$ is a down-set if for every $x\in U$ and $y\in X$ such that $y\leq x$, then $y\in U$. A basis for the topology on $X$ is $\set{\u x}_{x\in X}$. Note that $\u x$ is the minimal open set containing $x$. By a result of McCord \cite{mcc}, a poset $X$, viewed as a finite topological space, is weak homotopy equivalent to its order complex  $\k(X)$. Recall that $\k(X)$ is the simplicial complex of non-empty chains of $X$. Concretely, there is a weak homotopy equivalence $\mu:\k(X)\to X$ (i.e. a continuous map inducing isomorphisms in all homotopy groups). In particular $X$ and $\k(X)$ have the same homotopy groups and homology groups. In this paper we will not adopt the finite space point of view and we will understand the topology of $X$ via the simplicial complex $\k(X)$. 

  Given a finite regular CW-complex $L$, the face poset $\x(L)$ is the poset of cells of $L$ ordered by inclusion. Note that $\k(\x(L))=L'$,  the barycentric subdivision of $L$. We say that $X$ is regular if $X=\x(L)$ for some regular CW-complex $L$. Note that, in this case, for any $x\in X$, $\k(\u x)$ is homeomorphic to a sphere of dimension $h(x)-1$, since it is the barycentric subdivision of the boundary of the cell $x$. A poset $X$ is called $h$-regular  if for every $x\in X$, $\k(\u x)$ is homotopy equivalent to a sphere of dimension $h(x)-1$ (see \cite{min}). In particular, regular posets are $h$-regular.  We say that a poset $X$ is a model of CW-complex $K$ if $\k(X)$ is homotopy equivalent to $K$. In particular, $X$ is a model of $\k(X)$ and $\x(K)$ is a model of $K$ (and $K'$) for every regular CW-complex $K$.

In order to extend Forman's theory to all posets (not only regular or $h$-regular) we use height-type functions and admissible matchings. When $X$ is regular and the function is the usual height function, we recover Forman's theory. In the case that $X$ is $h$-regular and the function is the usual height, we obtain our previous results of \cite{min}.

\begin{definition}
Let $X$ be a poset. A function $f:X\to\R$ is a \it Morse function \rm if for every $x\prec y$ in $X$, $f(x)<f(y)$. 
\end{definition}

Note that the usual height function $h:X\to\R$ is a Morse function. Given $x\in X$, we denote $\ut x=\{y\in X, y<x\}=\u x-\set{x}$. The subposet $\ut x$ will be called the \it descending link \rm  of $x$.

\begin{definition}
	Let $f:X\to\R$ be a Morse function. An edge $(x,y)\in \h(X)$ is called \it admissible for $f$ \rm if it satisfies the following two conditions:
	\begin{enumerate}
		\item $\ut y-\set{x}$ is homotopically trivial (i.e. $\k(\ut y-\set{x})$ is contractible), 
		\item there is no $z\in X$ such that $f(x)<f(z)<f(y)$.
	\end{enumerate}
\end{definition}

\begin{remark}
	If $X$ is a regular poset, all edges of $X$ are admissible for the usual height function $h$. If $X$ is $h$-regular, an edge is admissible for the usual height function if it satisfies condition $(1)$ (cf. \cite[Lemma 2.8]{min}).
\end{remark}

Let $\h(X)$ be the Hasse diagram of a poset $X$. Following Chari \cite{cha}, given a matching $M$ on $\h(X)$,  we denote by $\h_M(X)$ the directed graph obtained from $\h(X)$ by reversing the orientations of the edges which are not in $M$. The matching is called a \it Morse matching \rm if the directed graph $\h_M(X)$ is acyclic. The points of $\h(X)$ not incident to any edge in $M$ are called \it critical. \rm  Given a Morse function $f:X\to\R$, a Morse matching $M$ is called \it admissible for $f$ \rm if every edge in $M$ is admissible for $f$. By the previous remark, any Morse matching on a regular poset is admissible for the usual height function.

\begin{ej} Figure \ref{nohreg} shows a non $h$-regular poset with a Morse function, which is not the usual height function, and an admissible Morse matching. The values of the function in each point are indicated in the figure and the edges of the matching are represented with dashed arrows. Note that the same Morse matching would not be admissible for the usual height function. In Example \ref{primerejemplo} we will see that the order complex of this poset is homotopy equivalent to $S^1\vee S^2$.

	\begin{figure}[h]
		\begin{displaymath}
			\xymatrix@C=10pt{
				 3\bullet  \ar@{-}[d]  \ar@{-}[drr]  \ar@{-}[drrrr]& &  3 \bullet \ar@{-}[d]  & &  & & 2\bullet   \ar@{-}[dllllll] \\
				1\bullet \ar@{-}[drr] & & 2\bullet  \ar@{-}[drr] \ar@{-}[d] \ar@{-}[dll]
				&& 2\bullet  \ar@{-->}[ull] \ar@{-}[dllll] \ar@{-}[dll] \\
				0\bullet  \ar@{-->}[u] && 0\bullet && 1\bullet \ar@{-->}[uurr]\\
			}
		\end{displaymath}
		\caption{An admissible Morse matching for a Morse function.\label{nuevo}}
	\end{figure}
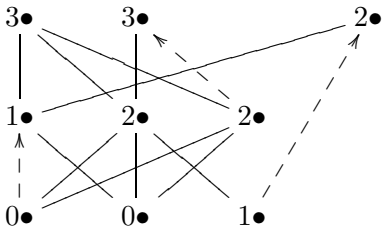

\end{ej}

\section{Main results}

 If $X$ is a regular poset and $f$ is the usual height function, the edges of the matching correspond to collapses of regular cells (cf. \cite{cha,for,for3}). In the general context this is not longer true and we use the following lemma to circumvent this problem. Recall that $\h_M(X)$ is the directed graph obtained from $\h(X)$ by reversing the orientations of the edges which are not in the matching $M$.

\begin{lemma}\label{mainlemma}
Let $f:X\to\R$ be a Morse function and let $M$ be an admissible Morse matching for $f$. If $x_0,\ldots,x_r$ is a directed path in $\h_M(X)$, then there is no $z\in X$ such that $f(x_0)<f(z)<f(x_r)$.
\end{lemma}
\begin{proof}
We proceed by induction on $r$. If $r=1$, there are two cases. If $x_0\prec x_1$, then $(x_0,x_1)\in M$ and the results follows from the definition of admissible edge. If $x_1\prec x_0$, the result follows from the definition of Morse function.

Given a directed path  $x_0,\ldots,x_r$, suppose there is a point $z\in X$ such that $f(x_0)<f(z)<f(x_r)$. By induction applied to the directed paths $x_0,\ldots,x_{r-1}$ and $x_{r-1},x_r$, we deduce that $f(z)=f(x_{r-1})$. Since $x_0,\ldots,x_r$ is a directed path and $f(x_{r-1})=f(z)<f(x_r)$, then $(x_{r-1},x_r)\in M$ and $x_{r-1}\prec x_{r-2}$, since $M$ is a matching. Therefore we get $f(x_0)<f(x_{r-1})<f(x_{r-2})$, which is a contradiction. 
\end{proof}

Our main result asserts that a Morse function together with an admissible Morse matching allows us to understand the topology of $X$ by giving a convenient order to its points and looking how the topology changes when we add criticial points. The topology of the poset eventually changes when we cross critical levels. If $L$ is a simplicial complex and $x$ is a vertex which is not in $L$, we denote by $xL$ the simplicial cone with apex $x$ and base $L$.

\begin{thm}\label{main}
Let $f:X\to \R$ be a Morse function and let $M$ be an admissible Morse matching for $f$. Then the order complex $\k(X)$ is homotopy equivalent to a CW-complex constructed by inductively gluing cones over the (order complexes of) descending links of the critical points. Concretely, if $t_0<t_1<\ldots<t_r \in\R$ are the values of $f$ and $C_i$ denotes the set of critical points $x\in X$ with $f(x)=t_i$, there is a filtration of $X$ by subposets $X_0\subseteq \ti X_0\subseteq X_1\subseteq \ldots \subseteq X_i\subseteq \ti X_i\subseteq \ldots \subseteq X_r=X$, where $\k(X_i)=\k(\ti X_{i-1})\bigcup_{x \in C_i} x\k(\ut x)$, and $\k(X_i)\subseteq \k(\ti X_i)$ is a strong deformation retract for all $i$.
\end{thm}

\begin{proof}
Let $x\in X$ be a source  node of $\h_M(X)$ with maximum $f$-value (i.e. $f(x)\geq f(z)$ for every source node $z$). If $x$ is a maximal point of $X$, then it is critical. Note that $\k(X)=\k(X-\set{x})\cup x\k(\ut x)$ with $\k(X-\set{x})\cap x\k(\ut x)=\k(\ut x)$. That is, $\k(X)$ is obtained from $\k(X-\set{x})$ by attaching the cone $x\k(\ut x)$ with apex $x$ over the complex of the descending link of $x$. Since $x$ is a critical maximal point, if we remove it from $X$, the Morse matching $M$ restricts to an admissible Morse matching for the Morse function $f$ on the subposet $X-\set{x}$. Note also that removing the point $x$ does not affect the descending links of the remaining points.

If $x$ is not a maximal point, there exists $y\in X$ with $x\prec y$. Since $x$ is a source node, it follows that $(x,y)\in M$ and that there is no $z\neq y$ such that $x\prec z$. Suppose $y$ is not a maximal point of $X$. Then there is an element $w\in X$ such that $x\prec y\prec w$. Since $(x,y)\in M$, then $(y,w)\notin M$. Let $s$ be a source node of $\h_M(X)$ such that $s=x_0,\ldots,x_r=w$ is a directed path in $\h_M(X)$. By Lemma \ref{mainlemma}, $f(s)\not< f(y)$ and then, $f(x)< f(y) \leq f(s)$ which contradicts the maximality of the $f$-value of $x$. It follows that $y$ is a maximal point of $X$. Since $y$ is the unique point of $X$ which covers $x$, then  $\k(X-\set{x})\subset \k(X)$ is a strong deformation retract (see for example \cite[Prop. 1.3.4]{b1}). This kind of points are called beat points, and the deletion of such a point does not affect the homotopy type.

On the other hand, the descending link of the point $y$ in the poset $X-\set{x}$ is $\ut y^{X-\set{x}}=\ut y^X-\set{x}$, which is homotopically trivial since the edge is admissible. Here we write   $\ut y^{X-\set{x}}$ and $\ut y^X$ to distinguish whether these subposets are considered in $X-\set{x}$ or in $X$. Since $\k(X-\set{x})$ can be obtained from $\k(X-\set{x,y})$ by gluing a cone over the contractible descending link $\k(\ut y^{X-\set{x}})$, then $\k(X-\set{x,y})\simeq \k(X-\set{x})\simeq \k(X)$, which means that we can remove both points without altering the homotopy type of the order complex. Moreover, since $y$ is a maximal point and $x$ is only covered by $y$, the  Morse matching $M$ restricts to an admissible Morse matching for the Morse function $f$ on the subposet $X-\set{x,y}$, and the removing of the points $x,y$ does not affect the descending links of the remaining points.

We can repeat the argument above until we remove all the points of the poset. Note that the last point that we remove is a critical point of height $0$. This procedure allows us to define the following. Let $t_0<t_1<\ldots<t_r \in\R$ be the values of $f$ in increasing order. We define a filtration of $X$ by subposets
$$X_0\subseteq \ti X_0\subseteq X_1\subseteq \ti X_1\subseteq \ldots \subseteq X_i\subseteq\ti X_i\subseteq\ldots \subseteq X_r=X$$
where $X_0$ is the subposets of critical points with $f$-value $t_0$ (note that this is a discrete subposet), $\ti X_0$ is obtained from $X_0$ by adding the points $x,y\in X$ such that $(x,y)\in M$ and $f(x)=t_0$, in the inverse ordering that they were removed. In general, $X_i$ is the subposet obtained from $\ti X_{i-1}$ by adding the critical points of $f$-value $t_i$ and $\ti X_i$  is obtained from $X_i$ by adding the points $x,y$ such that $(x,y)\in M$ and $f(x)=t_i$ (in the inverse ordering that they were removed). Note that $X_i\subseteq \ti X_i$ is a strong deformation retract at the level of complexes, and therefore $\k(X_i)$ is homotopy equivalent to $\k(X_{i-1})\bigcup_{x \in C_i} x\k(\ut x)$, where $C_i$ is the set of critical points of $f$-value $t_i$. This follows from the gluing theorem for homotopy equivalences (see for example \cite[Section 7.4]{bro}). The cone over the descending link of each critical point is glued via the deformation retraction $\k(X_{i-1})\subseteq \k(\ti X_{i-1})$.
 \end{proof}

 The Morse function provides a convenient ordering for adding the points one by one. Note that the cones corresponding to critical points with the same $f$-value can be glued in any order (or all of them at the same time) (cf. \cite{zar}).

\begin{remark}
	With the notation of the proof of Theorem \ref{main}, if the descending links of all critical points of $f$-value $t_i$ are homotopically trivial, then $\k(X_i)$ is homotopy equivalent to $\k(X_{i-1})$. As an immediate consequence, if the descending links of all critical points of height greater than $0$ are homotopically trivial, then $\k(X)$ is homotopy equivalent to a discrete space (the critical points of height $0$).
\end{remark}

\begin{remark}
If the descending links of the critical points are weak homotopy equivalent to spheres, then $\k(X)$ is homotopy equivalent to a CW-complex with one cell for each critical point. So, in particular, Theorem \ref{main} extends the theory of Forman for regular CW-complexes and our previous results on $h$-regular posets \cite{min}. 
\end{remark}

\begin{ej}\label{primerejemplo} The poset shown in Figure \ref{nuevo} is a model of $S^1\vee S^2$ (i.e. its order complex is homotopy equivalent to $S^1\vee S^2$). Note that the subposet $X_0$ consists of a single point (the unique critical point of minimum $f$-value), $X_1$ is homotopically trivial since we do not cross a critical level, $X_2$ is a model of $S^1\vee S^1$ since it is obtained from $X_1$ by a adjoining a cone over three different points. Finally $X=X_3$ is a model of a complex obtained from $X_2$ by attaching a cone over two copies of $S^1$, one of these copies is a boundary of a model of a disc in $\ti X_2$ (and this creates a $2$-sphere), the other is one of the $1$-spheres in $S^1\vee S^1\simeq X_2$ (then the cone homotopically kills that copy of $S^1$). This shows that $X$ is a model of $S^1\vee S^2$. 
	\end{ej}
 
\begin{ej} Figure \ref{nohreg} shows a non $h$-regular poset with a Morse function, which is not the usual height function, and an admissible Morse matching. The same Morse matching would not be admissible for the usual height function.  An easy inspection of the poset shows that it is a model of $S^2$. The descending link of the unique critical point with $f$-value $5$ is a model of $S^1$, and the subposet $X_2$ is clearly contractible. Note that, since there are no critical points with $f$-values between $2$ and $5$, and $\k(X_2)$ is contractible, $\k(X)=\k(X_5)$ is homotopy equivalent to the suspension of the descending link of the critical point with $f$-value $5$.

	\begin{figure}[h]
		\begin{displaymath}
			\xymatrix@C=5pt{
				& 4\bullet  \ar@{-}[drr] & & & 5 \bullet \ar@{-}[ddllll] \ar@{-}[dr]  & &  & 5\bullet  \ar@{-}[dllll]  \\
				& & & 2\bullet \ar@{-}[dr] \ar@{-}[ddrrrr] & & 4\bullet  \ar@{-->}[urr] \ar@{-}[dl] \ar@{-}[ddrr]\\
				3\bullet \ar@{-->}[uur] \ar@{-}[d]\ar@{-}[drrr] \ar@{-}[drrrrrrrrr]&&&& 1 \bullet  \ar@{-}[dl] 
				&&\\
				0\bullet \ar@{-}[uurrr] &&& 0\bullet  
				& & & & 0\bullet  && 0\bullet \ar@{-->}[ulllll]\\
			}
		\end{displaymath}
		\caption{An admissible Morse matching for a Morse function on a model of $S^2$.\label{nohreg}}
	\end{figure}
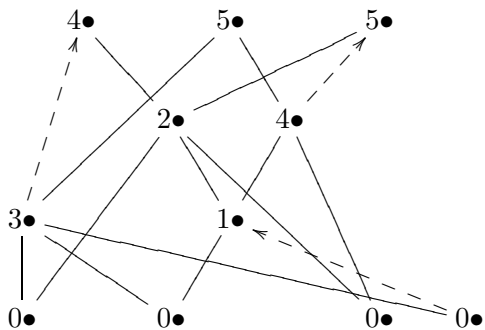

\end{ej}

\begin{ej}
The posets of Figure \ref{suspensiones} admit admissible Morse matchings for Morse functions with only two critical points: one minimal point and one maximal point $x$. This implies, in both cases, that $\k(X-\set{x})$ is contractible, and therefore $\k(X)=\k(X-\{x\})\cup x\k(\ut x)$ is homotopy equivalent to the suspension of the descending link $\Sigma (\k(\ut x))$. The poset on the left is a model of $\Sigma(S^1\vee S^1)=S^2\vee S^2$ and the poset on the right is a model of  $\Sigma(S^1\vee S^2)=S^2\vee S^3$.

	\begin{figure}[h]
		\begin{displaymath}
			\xymatrix@C=13pt{
&  &  & \\
2\bullet \ar@{-}[d] & 2\bullet^x \ar@{-}[dl] \ar@{-}[d] \ar@{-}[dr] & 2\bullet \ar@{-}[d]\\
1\bullet \ar@{-}[drr] & 1\bullet \ar@{-}[dl] \ar@{-}[dr] \ar@{-->}[ur] & 1\bullet \ar@{-}[dll] \ar@{-}[d] \ar@{-->}[ull]\\
0\bullet \ar@{-->}[u] &  & 0\bullet \\}
\hspace{3cm}	
		\xymatrix@C=13pt{
3\bullet^x  \ar@{-}[d]  \ar@{-}[drr] & & 3\bullet  \ar@{-}[dll] \ar@{-}[d]\\
2\bullet \ar@{-}[dr] & 2\bullet \ar@{-}[dl] \ar@{-}[d] \ar@{-}[ddr] \ar@{-->}[ur] & 1\bullet \ar@{-}[dd]\\
1\bullet \ar@{-}[d]  \ar@{-}[dr] \ar@{-->}[u] & 1\bullet \ar@{-}[d]  & \\
0\bullet \ar@{-->}[ur] & 0\bullet  \ar@{-->}[uur]  & 0\bullet \\}
		\end{displaymath}
		\caption{Posets with two critical points.\label{suspensiones}}
	\end{figure}
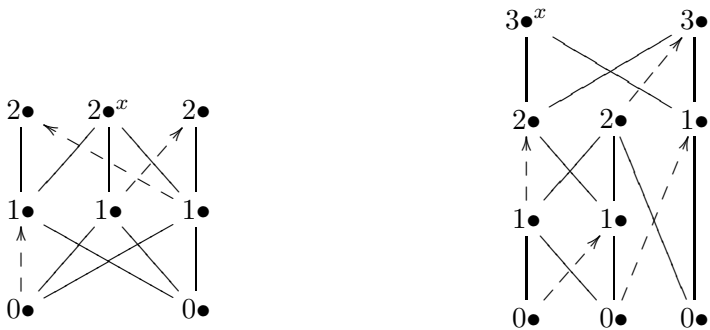

\end{ej}
\begin{definition}
	We say that two Morse functions $f,g:X\to\R$ are equivalent if for every $x,y\in X$, $f(x)<f(y)$ if and only if $g(x)<g(y)$. 
\end{definition}

\begin{remark}\label{equivalent}
	Note that if $f,g$ are equivalent Morse functions, any admissible Morse matching for $f$ is also admissible for $g$. Also, the subposets $X_i$ of the filtration in Theorem \ref{main} are the same for both functions.This means that one can replace a given function by an equivalent one without affecting the critical points and the filtration. This is convenient, for example, when working with the join (or ordinal sum) of posets, as we will see below.
\end{remark}

\section{Joins, opposite posets and relative Morse theory}

\subsection*{Morse theory on joins}

Recall that the simplicial join of two (disjoint) simplicial complexes $K,T$ is $K* T=K\cup T\cup\set{\sigma\cup\tau\ | \ \sigma\in K, \ \tau\in T}$.  The join (or ordinal sum)  $X\circledast Y$ of two posets $X$ and $Y$ is the disjoint union $X\sqcup Y$ keeping the giving ordering within $X$ and $Y$ and setting $x\leq y$ for every $x\in X$ and $y\in Y$. Clearly, $\k(X\circledast Y)=\k(X)* \k(Y)$ (see \cite{b1}). Note that the face poset $\x(K*T)$ of a join of simplicial complexes is not the join of the face posets $\x(K)\circledast \x(T)$. In general  $\x(K)\circledast \x(T)$ is much smaller than the face poset of the join. More important, the join of two regular posets is not longer regular, which implies that Forman's Morse theory cannot be applied to joins of posets. The same happens with $h$-regular posets. However, in our more general setting, given Morse functions and admissible Morse matchings on two posets $X,Y$, there exists an induce Morse function and an admissible Morse matching on the join $X\circledast Y$. This allows us to describe the topology of the join in terms of the critical points of $X$ and $Y$.

Given Morse functions $f:X\to \R$ and $g:Y\to\R$ with admissible Morse matchings $M$ and $N$, we first replace $g$, if necessary, by an equivalent Morse function (which we still denote $g$) such that $f(x)<g(y)$ for all $x\in X, y\in Y$. This can be done by adding a constant to $g$. By Remark \ref{equivalent} this does not affect the admissibility of the Morse matching $N$ neither the critical points and the filtration. We consider the function $f\sqcup g:X\circledast Y\to \R$ which is $f$ on the points of $X$ and $g$ on $Y$. Note that it is a Morse function. The matching $M \sqcup N$ is clearly a Morse matching, and it is admissible (for $f\sqcup g$)  because, for every $(x,x')\in M$, $\ut {x'}^{X\circledast Y}-\set{x}=\ut {x'}^{X}-\set{x}$ and for every $(y,y')\in N$,  $\ut {y'}^{X\circledast Y}-\set{y}= X\circledast (\ut {y'}^{Y}-\set{y})$, which is homotopically trivial since $\ut {y'}^{Y}-\set{y}$ is (see for example \cite{b1}). The critical points of the induced matching are the union of the critical points of $M$ and $N$, the descending links of the critical points in $X$ are  $\ut {x}^{X\circledast Y}=\ut {x}^{X}$ but the descending links of the critical points in $Y$ are  $\ut {y}^{X\circledast Y}= X\circledast \ut {y}^{Y}$. In some cases, the induced matching $M\sqcup N$ can be refined and one can obtain an admissible matching (for the same function $f\sqcup g$) with fewer critical points, as the following example shows.

\begin{ej}
	Figure \ref{join} shows admissible Morse matchings for the height functions on a model $X$ of $S^1\vee S^1$ and a model $Y$ of $S^1$, and a refinement of the induced Morse matching on the join $X\circledast Y$ (a model of $S^3\vee S^3$) with only two critical points. The descending link of the maximal critical point in $X\circledast Y$ is a model of $S^2\vee S^2$.
	
	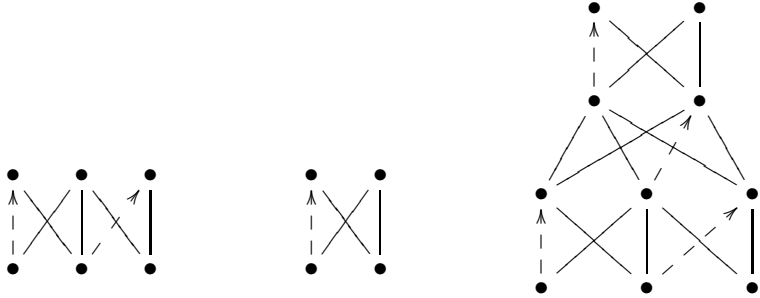
\begin{figure}[h]
		\begin{displaymath}
			\xymatrix@C=14pt{
				&  &  & \\
				& & &\\
				\bullet  \ar@{-}[dr]& \bullet \ar@{-}[dl] \ar@{-}[d] \ar@{-}[dr] & \bullet \ar@{-}[d]\\
				\bullet \ar@{-->}[u] & \bullet \ar@{-->}[ur] & \bullet \\}
			\hspace{1cm}	
			\xymatrix@C=14pt{
				& & \\
				& & \\
				\bullet  \ar@{-}[dr]& \bullet \ar@{-}[dl] \ar@{-}[d]  \\
			 \bullet \ar@{-->}[u] & \bullet \\}
		\hspace{1cm}	
			\xymatrix@C=8pt{
		&	\bullet  \ar@{-}[drr]& & \bullet \ar@{-}[dll] \ar@{-}[d]  \\
		& \bullet \ar@{-->}[u] \ar@{-}[dl] \ar@{-}[dr] \ar@{-}[drrr] & & \bullet \ar@{-}[dlll] \ar@{-}[dr]\\
			\bullet  \ar@{-}[drr]& & \bullet \ar@{-}[dll] \ar@{-}[d] \ar@{-}[drr] \ar@{-->}[ur]& &\bullet \ar@{-}[d]\\
			\bullet \ar@{-->}[u] & & \bullet \ar@{-->}[urr] & & \bullet \\}
		\end{displaymath}
		\caption{Refinement of the induced Morse matching on a join.\label{join}}
	\end{figure}
\end{ej}

\subsection*{Opposite posets}

Recall that the opposite poset $X^{op}$ of $X$ consists of the same underlying set but with the inverse order. Clearly, $\k(X^{op})=\k(X)$. In order to investigate the topology of $\k(X)$, sometimes it is convenient to analyze Morse matchings and functions on  $X^{op}$ instead of $X$, since it may admit matchings and functions with fewer critical points or critical points with nicer descending links. Note that in general the opposite of a regular posets is not longer regular, so we cannot apply Forman's theory in $X^{op}$.

\begin{ej}
	Figure \ref{opposite} shows a poset $X$ (on the left) and its opposite $X^{op}$ (on the right). Note that an optimal admissible Morse matching on $X$ (i.e. an admissible matching for some Morse function with the fewest number of critical points) has $4$ critical points. On the other hand, $X^{op}$ admits an admissible Morse matching for the height function with only two critical points. The descending link of the maximal critical point in $X^{op}$ is a model of $S^1$, which implies that $\k(X)=\k(X^{op})$  has the homotopy type of $S^2$.

	 	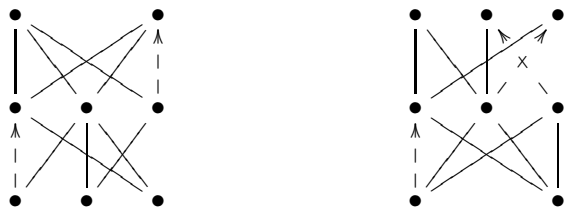
\begin{figure}[h]
	 	\begin{displaymath}
	 		\xymatrix@C=15pt{
	 			\bullet  \ar@{-}[d]\ar@{-}[dr]\ar@{-}[drr]& &  \bullet \ar@{-}[dll] \ar@{-}[dl] \\
	 			\bullet \ar@{-}[drr] & \bullet \ar@{-}[dl]\ar@{-}[d]\ar@{-}[dr] & \bullet \ar@{-->}[u]\ar@{-}[dl]\\
 			\bullet \ar@{-->}[u] & \bullet & \bullet}
	 		\hspace{3cm}	
	 		\xymatrix@C=15pt{
	 			\bullet  \ar@{-}[d]\ar@{-}[dr]& \bullet \ar@{-}[d]  & \bullet \ar@{-}[dll]  \\
	 			\bullet \ar@{-}[drr] & \bullet \ar@{-}[dl]\ar@{-}[dr] \ar@{-->}[ur] & \bullet \ar@{-}[dll]\ar@{-}[d]\ar@{-->}[ul]\\
	 			\bullet \ar@{-->}[u] & & \bullet \\}
	 	\end{displaymath}
	 	\caption{A poset $X$ on the left and its opposite $X^{op}$ on the right with a matching with fewer critical points.\label{opposite}}
	 \end{figure}
 \end{ej}

\subsection*{Relative Morse theory} We extend the results of the previous section to pairs $(X,A)$ with $X$ a (finite) poset and $A\subset X$ a down-set (or, equivalently, an open subset, if $X$ is viewed as a finite topological space). When $A$ is the empty set, we recover the result of Theorem \ref{main}. 
Given a down-set $A\subset X$, a Morse function on the pair $(X,A)$ is simply a function $f:A^c\to \R$, where $A^c=X\setminus A$ is the complement of $A$ in $X$ (viewed as a subposet of $X$), such that  for $x\prec y$ in $A^c$, $f(x)<f(y)$. Since $A$ is a down-set, the Hasse diagram of the complement $\h(A^c)$ is a subdiagram of $\h(X)$. An edge  $(x,y)\in \h(A^c)$ is called admissible for $f$ if it satisfies the same two conditions as in the absolute case:

\begin{enumerate}
	\item $\ut y^X-\set{x}$ is homotopically trivial, 
	\item there is no $z\in A^c$ such that $f(x)<f(z)<f(y)$.
\end{enumerate}
Similarly as before, we consider admissible Morse matchings $M$, i.e we require $\h_M(A^c)$ to be acyclic and every edge in $M$ be admissible for $f$. The proof of the relative case follows the ideas of Theorem \ref{main}.

\begin{thm}\label{relative}
Let $A\subset X$ be a down-set, $f:A^c\to \R$ a Morse function, and $M$ an admissible Morse matching for $f$. Then the order complex $\k(X)$ is homotopy equivalent to a CW-complex constructed from $\k(A)$ by inductively gluing cones over the descending links of the critical points. Concretely, if $t_0<t_1<\ldots<t_r \in\R$ are the values of $f$ and $C_i$ denotes the set of critical points $x\in A^c$ with $f(x)=t_i$, there is a filtration of $X$ by subposets
$A=\ti X_{-1}\subseteq X_0\subseteq \ti X_0\subseteq X_1\subseteq \ldots \subseteq X_i\subseteq \ti X_i\subseteq \ldots \subseteq X_r=X$, where $\k(X_i)=\k(\ti X_{i-1})\bigcup_{x \in C_i} x\k(\ut x)$, and $\k(X_i)\subseteq \k(\ti X_i)$ is a strong deformation retract for all $i$.
\end{thm}  

\begin{proof}
	We follow the proof of Theorem \ref{main}. Let $x\in A^c$ be a source node of $\h_M(A^c)$ with maximum $f$-value. If $x$ is a maximal point of $X$, then it is critical, and $\k(X)=\k(X-\set{x})\cup x\k(\ut x)$ with $\k(X-\set{x})\cap x\k(\ut x)=\k(\ut x)$. Since $x$ is a critical maximal point, if we remove it from $X$, the Morse matching $M$ restricts to an admissible Morse matching for the Morse function $f$ on the subposet $A^c-\set{x}$. 
	
	If $x$ is not a maximal point, there exists $y\in X$ with $x\prec y$. Since $x\in A^c$ and $A$ is a down-set, then $y\in A^c$, and since $x$ is a source node, then $(x,y)\in M$ and there is no $z\neq y$ such that $x\prec z$. Suppose $y$ is not a maximal point of $X$. Then there is an element $w\in X$ such that $x\prec y\prec w$. Again, since $A$ is down-set, then $w\in A^c$. Since $(x,y)\in M$, then $(y,w)\notin M$. Let $s$ be a source node of $\h_M(A^c)$ such that $s=x_0,\ldots,x_r=w$ is a directed path in $\h_M(A^c)$. By Lemma \ref{mainlemma}, $f(s)\not< f(y)$ and then, $f(x)< f(y) \leq f(s)$ which contradicts the maximality of the $f$-value of $x$. It follows that $y$ is a maximal point of $X$. Since $y$ is the unique point of $X$ which covers $x$, then  $\k(X-\set{x})\subset \k(X)$ is a strong deformation retract.	On the other hand, the descending link of the point $y$ in the poset $X-\set{x}$ is $\ut y^{X-\set{x}}=\ut y^X-\{x\}$, which is homotopically trivial since the edge is admissible. Since $\k(X-\set{x})$ can be obtained from $\k(X-\set{x,y})$ by gluing a cone over the contractible descending link $\k(\ut y^{X-\set{x}})$, then $\k(X-\set{x,y})\simeq \k(X-\set{x})\simeq \k(X)$ and we can remove both points without modifying the homotopy type of the order complex. Moreover, since $y$ is a maximal point and $x$ is only covered by $y$, the  Morse matching $M$ restricts to an admissible Morse matching for the Morse function $f$ on $A^c-\set{x,y}$, and the removing of the points $x,y$ does not affect the descending links of the remaining points.
	
	We repeat the argument until we remove all the points of $A^c$. Now let $t_0<t_1<\ldots<t_r \in\R$ be the values of $f$ in increase order. The filtration
	$$A=\ti X_{-1}\subseteq X_0\subseteq \ti X_0\subseteq X_1\subseteq \ti X_1\subseteq \ldots \subseteq X_i\subseteq\ti X_i\subseteq\ldots \subseteq X_r=X$$
	is defined similarly as before: $X_0$ is obtained from $A$ by adding the critical points of $A^c$ with minimum $f$-value, 
 $\ti X_0$ is obtained from $X_0$ by adding the points $x,y\in A^c$ such that $(x,y)\in M$ and $f(x)=t_0$, in the reverse order that they were removed. In general, $X_i$ is the subposet obtained from $\ti X_{i-1}$ by adding the critical points of $f$-value $t_i$ and $\ti X_i$  is obtained from $X_i$ by adding the points $x,y$ such that $(x,y)\in M$ and $f(x)=t_i$ (in the reverse order that they were removed). Note that for every $x\in C_i$, $\k(\ti X_{i-1}\cup\{x\})=\k(\ti X_{i-1})\cup x\k(\ut x^X)$ since $A$ is a down-set, and
 $\k(\ti X_{i-1})\cap x\k(\ut x^X)=\k(\ut x^X)$.
\end{proof}

\begin{remark}
	In contrast to the absolute case, it is possible to find relative Morse matchings without critical points. In that case, the inclusion $\k(A)\subset \k(X)$ is a homotopy equivalence. In particular, one can find relative Morse matchings without critical points of height $0$. The following example shows a relative Morse matching with a unique critical point of height $2$.
\end{remark}

\begin{ej}
	The poset of Figure \ref{proj} is a model of the projective plane $\R P^2$ (see \cite{b1}). The down-set $A$ is marked with dotted arrows and is a model of $S^1$. The (relative) Morse function is the usual height function and the admissible matching is marked with dashed arrows. There is only one critical point $x$ of height 2 whose descending link is a model of $S^1$. This shows that $\k(X)$ is homotopy equivalent to a CW-complex obtained from $\k(A)$ by attaching a $2$-cell.

	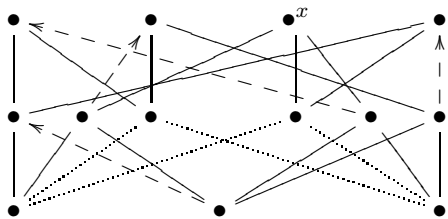
\begin{figure}[h]
		\begin{displaymath}
			\xymatrix@C=14pt{
				\bullet \ar@{-}[d]  \ar@{-}[drr]&  & \bullet \ar@{-}[d] \ar@{-}[drrrr] &  & \bullet^{x} \ar@{-}[dlll]  \ar@{-}[d] \ar@{-}[dr] & & \bullet  \ar@{-}[dllllll] \ar@{-}[dll]\\
				\bullet  \ar@{-}[d] & \bullet \ar@{-->}[ur] \ar@{-}[dl] \ar@{-}[drr] & \bullet  \ar@{.}[dll] \ar@{.}[drrrr] & &
				\bullet  \ar@{.}[dllll] \ar@{.}[drr] & \bullet  \ar@{-}[dll] \ar@{-}[dr]  \ar@{-->}[ulllll] & \bullet  \ar@{-}[d]  \ar@{-}[dlll]  \ar@{-->}[u]\\
				\bullet & & & \bullet  \ar@{-->}[ulll] & & & \bullet
				\\}
		\end{displaymath}
		\caption{A relative Morse matching on a model of $\R P^2$. The down-set $A$ is marked with dotted arrows. \label{proj}}
	\end{figure}

\end{ej}

\end{document}